\documentclass[12pt]{amsart}

\usepackage{amsmath}
\usepackage{amssymb}
\usepackage{latexsym}
\newtheorem{theorem}{Theorem}[section]

\newtheorem{prop}[theorem]{Proposition}

\newenvironment{proof*}{\vskip 2mm\noindent {}}{\hfill $\Box$ \vskip 2mm}
\numberwithin{equation}{section}
\newcommand{\C}{{\mathbb{C}}}

\newcommand{\D}{{\mathbb{D}}}
\newcommand{\eit}{{e^{i\theta}}}
\newcommand{\eitu}{{e^{i\theta_1}}}
\newcommand{\eitd}{{e^{i\theta_2}}}

\begin{document}

\title[Non cyclic in the bidisc]
{Non cyclic functions in the Hardy space of the bidisc with arbitrary decrease}

\author{Xavier Massaneda}
\address{Facultat de Matem\`atiques, Universitat de Barcelona, Catalonia}
\email{xavier.massaneda@ub.edu}

\author{Pascal J. Thomas}
\address{Universit\'e de Toulouse, UPS, INSA, UT1, UTM, Institut de Math\'ematiques de Toulouse, F-31062 Toulouse, France}
\email{pascal.thomas@math.univ-toulouse.fr}


\thanks{This work was carried out while both autors were participating in the Semester in Complex Analysis and Spectral Theory of the Centre de Recerca Matem\`atica at the Universitat Aut\`onoma de Barcelona.}

\begin{abstract}
We construct an example to show that no condition of slow decrease 
of the modulus of a function is sufficient to make it cyclic in the Hardy space
of the bidisc.  This is similar to what is well known in the case of the Hardy space 
of the disc,
but in contrast to the case of the Bergman space of the disc. 
\end{abstract}

\maketitle

\section{Background}

A cyclic vector $f$ for a given operator $\phi$ from a topological vector space $X$
to itself is one such that $ {\rm Span}\{\phi^n(f), n \in \mathbb Z_+\}$ is dense
in $X$.  In the classical framework of the Hardy space $H^2(\mathbb D)$ on the unit 
disk with the shift operator $\phi(f)(z):= zf(z)$, a cyclic function 
is one that verifies 
that $\C[Z]f$ is dense in $H^2(\mathbb D)$, where $\C[Z]$ denotes the set of 
holomorphic polynomials.  The characterization of those as outer functions goes back 
to Beurling, see \cite{Ga} or \cite{Ko} for details. For any $f \in H^2(\D)$,
$$
f (\zeta) := B(\zeta) \exp 
\left(  \int_0^{2\pi} \frac{\eit +\zeta}{\eit -\zeta} \ \left( g(\theta) d\theta-d\mu (\theta)\right)\right),
$$
where $B$ is a Blaschke product, $e^g \in L^2(\partial \D)$ and $\mu$ is a positive singular measure.  The factor containing only $\mu$ is called a singular inner function.  
The function $f$
is called \emph{outer} when $\mu=0$ and $B=1$ (or, equivalently, $\log|f|$ is
the Poisson integral of its boundary values). 

We are interested in the situation in $H^2(\D^2)$, where 
$\D^2 := \D \times \D \subset \C^2$, and 
$$
H^2(\D^2) := \left\{ f \in {\rm Hol}(\D^2) : \|f\|_2^2:=\sup_{r<1}
\int_0^{2\pi} \int_0^{2\pi} |f(r \eitu, r \eitd)|^2 d\theta_1 d\theta_2 < \infty \right\}.
$$
A \emph{cyclic} function is $f\in H^2(\D^2)$ such that $\C[Z_1,Z_2]f$ is dense in $H^2(\D^2)$.

Rudin \cite{Ru} proved that
the analogous equivalence between being outer (in an appropriate sense) and cyclic
fails in the case of the polydisk. A lot of work on analogous 
problems was later carried out by Hedenmalm, see for instance \cite{He}.

Since functions with the same modulus are cyclic or not
simultaneously --- because if $f_1$ is cyclic
and $|f_2(z)| \ge |f_1(z)|$, for any $z \in \D^2$, 
then $f_2$ is cyclic --- it is natural to look for conditions
on the size of $f$ that would be necessary or sufficient for cyclicity. 
In general, $|f|$ cannot be allowed to vanish, nor decrease too fast near the
boundary, in order for $f$ to be cyclic. 

One type of necessary condition can be obtained by restricting functions to the
diagonal.  Rudin \cite[p. 53]{Ru} noticed that the map 
$f \mapsto Rf(\zeta):=f(\zeta,\zeta)$ is bounded and onto from $H^2(\D^2)$
to the Bergman space $A^2(\D):= {\rm Hol}(\D) \cap L^2(\D)$ (this is a quite general
phenomenon, see \cite{HoOb} for instance).  Thus if $f$ is cyclic in $H^2(\D^2)$,
then $Rf$ is cyclic in $A^2(\D)$, that is to say, $\C[Z]Rf$ is dense in $A^2(\D)$.

Borichev \cite{Bo} proved that if $f\in A^2(\D)$ and
$|f|$ decreases slowly enough, then $f$ is cyclic in $A^2(\D)$.  More precisely,
we call \emph{weight function} a non increasing function $v$ from $(0,1]$
to $(0,\infty)$ such that $\lim_{t\to 0}v(t)=\infty$ and $v(t^2) \le C v(t)$ for some $C>0$.  The following is a special case of \cite[Theorem 2]{Bo}:
\begin{theorem}
If $|f(\zeta)| \ge \exp(-v(1-|\zeta|))$ for a weight function $v$ such that
$$
\int_0 \frac{v(t)^2}{t(\ln t)^2} dt <\infty,
$$
then $f$ is cyclic in $A^2(\D)$. 
\end{theorem}

It has been known for a long time that no sufficient condition of this 
type can hold for $H^2(\D)$; indeed, no non-trivial singular inner function
may be cyclic, but one can construct them with arbitrarily slow decrease.
On the other hand, in view of the sufficient condition in the Bergman
space case, one may wonder whether a similar sufficient condition
could hold for $H^2(\D^2)$.  The purpose of this note is to show
that it is not the case, as the example below shows.

We thank A. A. Aleksandrov for stimulating conversations on this topic.

\section{The example}

For $(z_1,z_2) \in \D^2$, let $\delta (z) := \max (1-|z_1|, 1-|z_2|)$ 
(this is comparable to the distance to the distinguished boundary of the bidisk
$(\partial \D)^2$).

\begin{prop}
For any weight function $v$, there exists $f \in H^2 (\D^2)$ 
such that $\log |f(z)| \ge -v(\delta(z))$, but $f$ is not cyclic
for $H^2 (\D^2)$.
\end{prop}

\begin{proof}
Let $\mu_v$ be a singular positive measure on $\partial \D$ such
that $\mu_v (I) \le c_0 |I|  v(|I|)$, for $c_0$ small enough, then the singular function 
$$
f_0 (\zeta) := \exp 
\left( - \int_0^{2\pi} \frac{\eit +\zeta}{\eit -\zeta} \ d\mu_v (\theta)\right)
$$
is not cyclic for $H^2(\D)$ and verifies 
$\log |f_0(z)| \ge -  v(1-|z|)$.

Now set $f(z):= f_0(z_1z_2)$.  Since 
$$
\max (1-|z_1|, 1-|z_2|) \le 1- |z_1z_2| \le 1-|z_1|+ 1-|z_2|,
$$
an appropriate modification of the constant $c_0$ yields
$\log |f(z)| \ge -v(\delta(z))$.

If $f_0(\zeta)= \sum_k \alpha_k \zeta^k$, then 
$f(z) =  \sum_k \alpha_k z_1^k z_2^k = \sum_{j,k} a_{j,k} z_1^j z_2^k$, 
with $\sum_{j,k} |a_{j,k}|^2 = \sum_k |\alpha_k|^2 < \infty$, so
$f\in H^2(\D^2)$ (or in this case we just might notice we have bounded functions).

In what follows, we freely use the fact that any function in $H^2(\D)$ or
$H^2(\D^2)$ admits almost everywhere defined non-tangential boundary values,
and that its norm is equal to the $L^2$ norm on the circle (resp. torus).

Suppose there exists a sequence of polynomials $P_n(z_1,z_2)$ such that
$$
\lim_n \int_0^{2\pi} \int_0^{2\pi}
\left| P_n (\eitu, \eitd) f (\eitu, \eitd) - 1 \right|^2 \,
d \theta_1  d \theta_2 = 0.
$$
Using the change of variables $\theta_1 = \theta$, $\theta_2 = \theta + \alpha$, we get
$$
\lim_n \int_0^{2\pi} 
\left( 
\int_0^{2\pi}
\left| P_n (\eit, e^{i (\theta + \alpha)}) f_0 (e^{i (2\theta + \alpha)}) - 1 \right|^2 \,
d \theta 
\right)
 d \alpha = 0,
$$
which implies that there exists an increasing sequence $n_k$ of integers
such that for almost any $\alpha$, 
$$
\lim_k 
\int_0^{2\pi}
\left| P_{n_k} (\eit, e^{i (\theta + \alpha)}) f_0 (e^{i (2\theta + \alpha)}) - 1 \right|^2 \,
d \theta 
= 0.
$$
For such an $\alpha$, let $g_\alpha (\zeta) := f_0 (e^{i  \alpha} \zeta^2)$.
This is again a singular function, but we have a sequence of polynomials 
$q_k (\zeta) := P_{n_k} (\zeta, e^{i  \alpha} \zeta)$ such that 
$\|q_k g_\alpha - 1 \|_{H^2(\D)} \to 0$, which is a contradiction. 

\end{proof}

\end{document}